\documentclass[a4paper,11pt]{amsart}
\usepackage[plainpages=false]{hyperref}
\usepackage{amsfonts,latexsym,amsthm}
\usepackage{amsmath,amssymb,latexsym,lscape,rawfonts}

\def\TT{{\mathbb{T}}}

\def\NN{{\mathbb{N}}}
\def\PP{{\mathbb{P}}} 

\def\ZZ{{\mathbb{Z}}}

\def\CC{{\mathbb{C}}}
\def\RR{{\mathbb{R}}} 
 
\def\TT{{\mathbb{T}}}

\newtheorem{thm}{Theorem}[section]

\newtheorem{cor}[thm]{Corollary}
\newtheorem{lemma}[thm]{Lemma}

\newtheorem{prop}[thm]{Proposition}
\newtheorem{rmk}[thm]{Remark}
\newtheorem{Def}[thm]{Definition}
\newtheorem{conj}[thm]{Conjecture}

\title{Calabi flow on projective bundles, I} 
\author{Hongnian Huang}
\date{}

\begin{document}

\begin{abstract}
In this paper, we obtain several a-priori estimates for the Calabi flow on projective bundles admitting the generalized Calabi constructions.
\end{abstract}
\maketitle

\section{Introduction}

In Kaehler geometry, there are mainly two conjectures concerning the behavior of the Calabi flow:

\begin{conj}[Chen]
\label{conj1}
The Calabi flow exists for all time.
\end{conj}

\begin{conj}[Donaldson]
\label{conj2}
If the Calabi flow exists for all time and there exists a constant scalar curvature Kaehler (cscK for short) metric in the Kaehler class, then the Calabi flow converges to a cscK metric.
\end{conj}

On one hand, there are many work concerning \autoref{conj2} recently, see \cite{Berman, H3, H4, St, He, LWZ}. On the other hand, to the best of the author's knowledge, our understanding of the long time existence of the Calabi flow is restricted into the following two cases:

\begin{enumerate}
\item  In the cases of Riemann surfaces and ruled manifolds, the Calabi flow equation is reduced of an ODE equation and its long time existence has been proved, see \cite{Chr, Chen1, Guan, Sz}. 

\item Chen and He \cite{ChenHe2, ChenHe3} prove the long time existence of the Calabi flow on toric Fano surfaces under the following two assumptions: 1. the initial metric is a toric invariant one. 2. the Calabi energy of the initial metric is less than an explicit constant. 

\item Joint with Feng \cite{FH}, we proved that the Calabi flow exists for all time in $\CC^2 / (\ZZ^2 + \sqrt{-1} \ZZ^2)$ starting from a torus invariant metric.
\end{enumerate}

In this paper, we study the long time existence problem of the Calabi flow on $(X, J) = P(E) \to S$, where $E = \mathcal{O} \oplus \mathcal{L}_1 \oplus \mathcal{L}_2$ and $\mathcal{L}_1, \mathcal{L}_2$ are holomorphic line bundles over a cscK manifold $(S, \omega_S)$. Apostolov, Calderbank, Gauduchon and T\o nnesen-Friedman \cite{ACGT} constructed the generalized Calabi metric (or the admissible metric) on $(X, J)$. Note that their construction of a generalized Calabi metric is in a Delzant polytope $P$ associated with $(X, J)$ by the moment map. Thus the techniques developed in toric geometry can be applied in $(X, J)$. By doing analysis on the generalized Calabi constructions in \cite{ACGT}, we obtain the a-priori interior estimates on $(X, J)$ of the Calabi flow. We also obtain the a-priori $C^0$-estimate of the Calabi flow in some cases by showing that the Sobolev constants are uniformly bounded along the Calabi flow. More explicitly, we have:

\begin{thm}[Interior esimates]
\label{interior}
Let $u(t), ~ t \in [0, T)$ be a one parameter group of symplectic potentials satisfying the Calabi flow equation. Then for any $\epsilon > 0$, there exists $C(\epsilon) > 0, ~ C(k, \epsilon) > 0, ~ \forall k \in \NN$ such that
\begin{align*}
C(\epsilon)^{-1} I_2 < (D^2 u_t(x)) < C(\epsilon) I_2,\\
|D^k u_t (x)| < C(k, \epsilon),
\end{align*}
where $I_2$ is a $2*2$ identity matrix and $x \in P_\epsilon$ which consists all the points whose Euclidean distance to $\partial P$ is at least $\epsilon$.
\end{thm}

\begin{rmk}
Our result can be extended to the more general cases: 

\begin{enumerate}

\item One can consider the admissible metric on $P(O \oplus E_1 \oplus E_2) \to S$, where $E_i$ are projectively flat bundle over $S$.

\item One can also consider the general admissible metrics where the $\CC \PP^2$ bundle is replaced by any toric surface.

\end{enumerate}

\end{rmk}

\begin{thm}[Uniform Sobolev control]

\label{thm_proj_Sob}
Let $(X, J) = \PP(O \oplus L_1 \oplus L_2) \to S$, where $(S, \omega_S)$ is a Riemann surface with constant scalar curvature $-1, 0$ or $1$. Let $C_S \ge 12 p_1$ be the constant controlling the Kaehler class where $p_1 = \deg(L_1) > \deg(L_2)$. Let $u_{FS}$ be the symplectic potential on the Delzant polytope $P$ corresponding to the Fubini-Study metric of $\CC \PP^2$. Suppose $u(t), t \in [0, T)$ is a one parameter family of symplectic potentials satisfying the Calabi flow equation with $u(0) = u_{FS}$, then the Sobolev constant of $u(t)$ is uniformly bounded for any smooth function $f \in C^\infty(\bar{P})$, i.e., there exists a constant $C > 0$ independent of $t$ such that

$$
\| f \|_{L^3(u(t))} \le C \left( \| f \|_{L^2(u(t))} + \| \nabla f \|_{L^2(u(t))} \right)
$$
\end{thm}

\begin{rmk}
One could easily extend our result to any general cscK manifold $(S, \omega_S)$ with appropriate $C_S$. 
\end{rmk}

The author has initiated another project on the study of the Calabi flow with uniform Sobolev bounds on toric manifolds \cite{H1}. By modifying the techniques developed in \cite{H1}, we obtain the following result:

\begin{cor}[$C^0$-esimate]
\label{c0}
Following the settings in Theorem \ref{thm_proj_Sob}, we have
$$
|u(t)|_{L^\infty} \le C,
$$
for any $t \in [0, T)$.
\end{cor}

{\bf Acknowledgement: } Our paper was initiated by the conversations with Professor Vestislav Apostolov during the time that the author was a postdoc in CIRGET, Montreal, Canada. The author would like to thank him for consistent support and encouragement. We also wish to thank Professor Paul Gauduchon for his interest in this paper.

\section{Notation and setup}

Let $(X, J, \omega)$ be a K\"ahler manifold with an integrable complex structure $J$ and a symplectic form $\omega$. Let $n$ be the complex dimension of $X$. In local holomorphic coordinates system
\begin{align*}
z_i, ~ i = 1, \ldots, n,
\end{align*}
by the $\partial \bar{\partial}$-lemma, we can write $\omega$ as
\begin{align*}
  \omega = \sqrt{-1} \frac{\partial^2 \varphi}{\partial z_i \partial \bar{z}_j} d z_i \wedge d \bar{z}_j,
\end{align*}
where $\varphi \in C^\infty(X)$. The K\"ahler metric is

\begin{align*}
  g = g_{i\bar{j}} d z_i \otimes d \bar{z}_j = \frac{\partial^2 \varphi}{\partial z_i \partial \bar{z}_j} d z_i \otimes d \bar{z}_j.
\end{align*}

The bisectional curvature can be expressed as 

\begin{align*}
  Rm_{i\bar{j}k\bar{l}} = - \frac{\partial^2 g_{i\bar{j}}}{\partial z_k \partial \bar{z}_l} + g^{s\bar{t}} \frac{\partial g_{s\bar{j}}}{\partial \bar{z}_l} \frac{\partial g_{i \bar{t}}}{\partial z_k},
\end{align*}

and the Ricci curvature and the scalar curvature are

\begin{align*}
  Ric_{i \bar{j}} =& - \frac{\partial^2}{\partial z_i \bar{\partial z}_j} \log (\det (g_{k\bar{l}})), \\
  R =& - 2 \triangle \log (\det (g_{i \bar{j}}).
\end{align*}

The Ricci form is defined as 
\begin{align*}
  \rho = \sqrt{-1} \partial \bar{\partial} \log (\det (g_{k \bar{l}}))
\end{align*}

Since the first Chern class $c_1(X)$ can be represented by $\frac{\rho}{2 \pi}$, the average of the scalar curvature is 
\begin{align*}
  \underline{R} = \frac{\int_{X} 2 \rho \wedge \frac{\omega^{n-1}}{(n-1)!}}{\int_X \frac{\omega^n}{n!}} = 4 \pi n \frac{c_1(X) \cdot \omega^{n-1}}{\omega^n}.
\end{align*}

We denote the set of {\it relative} K\"ahler potential as

\begin{align*}
   \mathcal{H} = \{ \psi \in C^\infty(X) ~|~ \omega_\psi = \omega + \sqrt{-1} \partial \bar{\partial} \psi > 0\}.
\end{align*}

The Calabi flow starting from $\psi$ will be a one parameter of {\it relative} K\"ahler potential $\psi(t)$ satisfying $\psi(0) = \psi$ and

\begin{align*}
  \frac{\partial \psi(t)}{\partial t} = R(\psi(t)) - \underline{R}.
\end{align*}

\subsection{Toric geometry}

Let $(X, J, \omega)$ be a toric manifold with a Delzant polytope $P$. Let $z$ be the moment map, i.e.,

\begin{align*}
  z : X \to \RR^n
\end{align*}
and $z(X) = P$. The preimage of $P^0$, i.e., the interior of $P$, is the $(C^*)^n$ orbit in $X$. We can write local holomorphic coordinates $\eta_i$ as

\begin{align*}
  \eta_i = \xi_i + \sqrt{-1} t_i, ~ \xi_i \in \RR, ~ t_i \in [0,2 \pi], \quad i = 1, \ldots, n
\end{align*}

such that the $T^n$-action is the translation of the $t_i$ variable. Moreover, we can express the toric invariant metric $g$ as follows:

\begin{align*}
  g = \frac{1}{2} \frac{\partial^2 \varphi}{\partial \xi_i \partial \xi_j} d z_i \otimes d \bar{z}_j,
\end{align*}

where $\varphi$ is a smooth, strictly convex function on $\RR^n$. Then the moment map $z(\eta)$ is

\begin{align*}
  (x_1, \ldots, x_n ) := z(\eta_1, \ldots, \eta_n) = \left( \frac{\partial \varphi}{\partial \xi_1}, \ldots, \frac{\partial \varphi}{\partial \xi_n} \right).
\end{align*}

By the Legendre transform, we obtain the symplectic potential as 

\begin{align*}
  u(x_1, \ldots x_n) = \sum_{i=1}^n x_i \xi_i - \varphi(\xi_1, \ldots, \xi_n).
\end{align*}

We can express the K\"ahler metric $g$ in the symplectic coordinates as follows:

\begin{align*}
  g = u_{ij} d x_1 d x_j + u^{ij} d t_i d t_j,
\end{align*}
where
\begin{align*}
  \left(u_{ij} \right) = \left( \frac{\partial^2 u}{\partial x_i \partial x_j} \right), \quad \left( u^{ij} \right) = \left( u_{ij} \right)^{-1}. 
\end{align*}

Then the volume of $X$ is 

\begin{align*}
  \frac{(2 \pi)^n}{n!} \int_P ~ d \mu,
\end{align*}
where $d \mu$ is the standard Lebesgue measure in $\RR^n$.\\

Recall that the scalar curvature of $g$ in the complex coordinate is

\begin{align*}
  R = - \sum_{i,j=1}^n \varphi^{ij} \left( \log \det (\varphi_{kl}) \right)_{ij},
\end{align*}
where
\begin{align*}
  \left( \varphi_{ij} \right) = \left( \frac{\partial^2 \varphi}{\partial \xi_i \partial \xi_j} \right), \quad \left( \varphi^{ij} \right) = \left( \frac{\partial^2 \varphi}{\partial \xi_i \partial \xi_j} \right)^{-1}. 
\end{align*}

Following Abreu's calculations \cite{A1}, we obtain the expression of the scalar curvature in the symplectic coordinates:

\begin{align*}
  R = - \sum_{i,j=1}^n u^{ij}_{~ij}.
\end{align*}

Also, the norm of the bisectional curvature is

\begin{align*}
  |Rm|^2 = \frac{1}{4} \varphi^{ii'} \varphi^{jj'} \varphi^{kk'} \varphi^{ll'} \left( - \varphi_{ijkl} + \varphi^{st} \varphi_{sjl} \varphi_{itk} \right) \left( - \varphi_{i'j'k'l'} + \varphi^{st} \varphi_{sj'l'} \varphi_{ti'k'} \right).
\end{align*}

Direct calculations show that (e.g. \cite{D2}) in the symplectic coordinates

\begin{align*}
  |Rm|^2 = \frac{1}{4} \sum_{i,j,k,l=1}^n u^{ij}_{~kl} u^{kl}_{~ij}.
\end{align*}

Note that $u$ also satisfies the Guillemin boundary condition. We recall that $P$ satisfies the following condition:

\begin{itemize}
  \item For any facet $P_i$ of $P$, there exists a unique inward primitive normal vector $\vec{v}_i$.

  \item For any vertex $v$, there exists exactly $n$ facets $P_{i_1}, \ldots, P_{i_n}$ meeting at $v$ and $\vec{v}_{i_k}, ~ k = 1, \ldots, n$ form a basis of $\ZZ^n$.
\end{itemize}

Suppose P has $d$-facets, then for each facet $P_i$, we choose a constant $c_i$ such that $l_i(x) = \langle x, \vec{v}_i \rangle + c_i$ vanishes on $P_i.$ We say that $u$ satisfies the Guillemin boundary conditions if

\begin{itemize}
  \item $u$ extends as a continuous function over $\partial P$ and its restriction to the interior of each face of $P$ is smooth and strictly convex. 

  \item 
    
    \begin{align*}
      u(x) = \frac{1}{2} \sum_{i=1}^d l_i(x) \ln(l_i(x)) + f(x),
    \end{align*}
    where $f(x)$ is a smooth function on $\bar{P}$.
\end{itemize}

Let $\mathcal{H}_s$ be the set of symplectic potentials, i.e.,

\begin{align*}
  \mathcal{H}_s = & ~ \{u ~ \mathrm{is ~ a ~ smooth, ~ strictly ~ convex ~ function ~ on ~} P \\
  & ~ \mathrm{and} ~ u ~ \mathrm{satisfies ~ the ~ Guillemin ~ boundary ~ conditions.}\}
\end{align*}

By the Legendre transform, a one parameter family $u(t)$ satisfies the Calabi flow equation if

\begin{align*}
  \frac{\partial u(t)}{\partial t} = \underline{R} - R(u(t)).
\end{align*}

\subsection{Projective bundles and the general Calabi constructions}

Let $(X, J) = \PP(O \oplus L_1 \oplus \cdots \oplus L_l) \to S$, where $L_i, ~ 1 \le i \le l$ are Hermitian line bundles over a cscK manifold $(S, \omega_S)$. Apostolov, Calderbank, Gauduchon and T\o nnesen-Friedman \cite{ACGT} constructed the admissible metrics on $(X, J)$ which is a generalization of the Calabi construction. In order to do that, they require the toric action $\TT^l$ to be rigid:

\begin{Def}
Let $g$ be a $\TT^l$-invariant Kaehler metric on $(X, J)$. Let $z : X \to \mathfrak{t}^*$ be the moment map of $\TT^l$, where $\mathfrak{t}$ is the Lie algebra of $\TT^l$. We say $\TT^l$ is a rigid toric action if for any $x \in X$, $\mathfrak{i}_x^* g$ depends only on $z(x)$, where
$
\mathfrak{i}: \mathbb{T} \rightarrow \mathbb{T} \cdot x \subset X
$
is the orbit map.
\end{Def}

Let the image of $X$ under the moment map $z$ be $P$ which is a Delzant polytope. Let $P^0$ be the interior of $P$. Let $\theta_i$ be a connection on $L_i$ such that its curvature is $p_i \omega_S$. Then on $z^{-1}(P^0)$, one can express the generalized Calabi metric (or the admissible metric) $g$ in the following way:

\begin{equation}
\label{cm}
g = ( \langle p_S, z \rangle + c_S) g_S +  \langle dz, G, dz \rangle + \langle \theta, H, \theta \rangle,
\end{equation}

where $p_S =  (p_1, \ldots, p_l ) \in \RR^l, ~ \theta = (\theta_1, \ldots, \theta_l), ~ G = Hess(u) = H^{-1},~ u$ is a symplectic potential of $(\CC \PP^l, P)$. $C_S$ is a constant parametrizing the Kaehler class and it satisfies $p(z) > 0$ on $\bar{P}$.

The scalar curvature of $u$ is :
$$
R(u) = \frac{Scal_S }{\langle p_S, z \rangle + c_S} 	 - \frac{1}{p(z)} \sum_{r,s=1}^l \frac{\partial^2}{\partial z_r \partial z_s} (p(z)u^{rs}),
$$
where $p(z)=(\langle p_\Sigma, z \rangle + c_\Sigma)^m, ~ m = \dim_{\CC} S $. 

The volume form is
$$
\omega^n = p(z) \omega_S^m \wedge \langle dz \wedge \theta \rangle^l,
$$
where $n = m + l$ is the dimension of $X$.

Let $u(t), t \in [0, T)$ be a one parameter family of symplectic potentials on $P$ satisfying the Calabi flow equation, i.e.,
$$
\frac{\partial u(t)}{\partial t} = \underline{R} - R(u).
$$ 

Note that the Calabi flow preserves the admissible K\"ahler class:

\begin{prop}
  Let $\varphi(t), t \in [0, T)$ be the Calabi flow starting from an admissible metric $\omega$ with $\varphi(0) = 0$, then $\omega(t) = \omega+ \sqrt{-1}\partial \bar{\partial} \varphi(t)$ is also an admissible metric for $t \in [0, T)$.
\end{prop}

\begin{proof}
  To prove that the Calabi flow preserves the admissible K\"ahler class, we need to adapt the proof of Theorem 3.2 in \cite{ChenHe} to the appropriate Banach spaces. 
  
  Let $E_0 = c^{0, \alpha}(\CC \PP^l)$ be the closure of the $\TT^l$-invariant smooth functions in $\CC \PP^l$ under the $C^{0,\alpha}$ H\"older norm, using the standard Fubini-Study metric. Also we let $E_1 = c^{4,\alpha}(\CC \PP^l)$, $E_\theta = c^{3,\alpha}(\CC \PP^l)$ and
  \begin{align*}
    V_\theta = \{ \varphi ~|~ \varphi \in c^{3,\alpha}(\CC \PP^l), ~ \lambda \omega < \omega(\varphi) < \Lambda \omega, ~ |\varphi|_{c^{3,\alpha}(\CC \PP^l)} < K \},
  \end{align*}
  where $\theta = \frac{3}{4}, \lambda, \Lambda, K$ are some positive constants and $\omega(\varphi) = \omega + \sqrt{-1} \partial \bar{\partial} \varphi$. 

Let $g$ be the Riemannian metric corresponding to $\omega$. Let $f$ be any $\TT^l$ invariant function on $\CC \PP^l$. Then $f$ can be seen as a function on $X$ through the moment map. We denote

  \begin{align*}
    A_g(\varphi) = 2 g^{i\bar{j}}g^{k\bar{l}} \nabla_i \nabla_{\bar{j}} \nabla_k \nabla_{\bar{l}} \varphi.
  \end{align*}

  One can check that
  \begin{align*}
    A : E_1 \to E_0.
  \end{align*}
  Now we let $f(\varphi) = R(\omega_\varphi) + A_{g_\varphi}(\varphi).$ In order to complete our proof, we only need to show that for any $\varphi, \psi \in V_\theta, ~ w \in E_1$, we have
  \begin{align*}
    |f(\varphi) - f(\psi)|_{c^{0,\alpha}(\CC \PP^l)} \le & C |\varphi - \psi|_{c^{3,\alpha}(\CC \PP^l)}, \\
    |A_{\varphi}(w) - A_{\psi}(w)|_{c^{0,\alpha}(\CC \PP^l)} \le & C |\varphi - \psi|_{c^{3, \alpha}(\CC \PP^l)} |w|_{c^{4,\alpha}(\CC \PP^l)}.
  \end{align*}

We denote the Banach spaces $c^{k, \alpha}(X)$ as the closure of the smooth functions on $X$ under the $C^{k, \alpha}$ norm of $\omega$. Since $\CC \PP^l$ is a totally geodesic submanifold of $(X, \omega)$, then for any two points $x_1, x_2 \in P$ such that $x_1$ is within the cut locus of $x_2$ in $X$, we have $d_{X}(x_1, x_2) = d_{\CC \PP^l}(x_1, x_2)$, where $d_X, d_{\CC \PP^l}$ are the distance functions with respect to $\omega$ and $\omega |_{\CC \PP^l}$ respectively. It implies that for $f \in C^\infty(\CC \PP^l)$, we have $\| f \|_{c^{0, \alpha}(X)} = \| f \|_{c^{0, \alpha}(\CC \PP^l)}$. Again, the fact that $\CC \PP^l$ is a totally geodesic submanifold of $(X, \omega)$ help us control $\| f \|_{C^k(X)}$ by $\| f \|_{C^k(\CC \PP^l)}$ for $k \le 4$ and vice versa. Similarly, we can control the $C^{k, \alpha}(X)$ norm of $f$ by $C^{k, \alpha}(\CC \PP^l)$ norm of $f$ for $k \le 4$ and vice versa.

  The rest of the proof is identical to the proof of Theorem 3.2 in cite.
\end{proof}

\section{Interior estimates}

In this section, we will obtain the interior estimates of the Calabi flow $u(t), t \in [0, T)$ on $(X, J) = \PP(O \oplus L_1 \oplus L_2) \to S$, where $(S, \omega_S)$ is a cscK manifold. Our techniques are based on the earlier work \cite{CHS}. Let $P_\epsilon$ be the set of points in $P$ whose Euclidean distance to $\partial P$ is at least $\epsilon$. Our first result is the following:

\begin{prop}
\label{distance}
Let $u(t), t \in [0, T)$ be a Calabi flow on $(X, J)$. For any $\epsilon > 0$, there exists a constant $C(\epsilon) > 0$ independent of $t$ such that 
$$
d_t (P_\epsilon, P_{2 \epsilon}) > C(\epsilon)
$$
for any $t \in [0, T)$, where $d_t$ denotes the Riemannian distance of the metric $u(t)_{ij} d z_i \otimes d z_j$ on $P$.
\end{prop}

To prove \autoref{distance}, we need to obtain some preliminary results.

\begin{lemma}
\label{l2-estimate}
There exists a constant $C > 0$ independent of $t$ such that for any $t \in [0, T)$, 
$$
\int_P u^2(t) ~ d \mu < C.
$$
\end{lemma}

\begin{proof}
Let $T_0 = T/2$. Since the Calabi flow decreases the distance \cite{CC}, one obtains that
$$
\int_P (u(t+T_0) - u(t))^2 ~ p(z) d \mu
$$
is decreasing for $t \in [0, T/2)$. Then there exists a constant $C > 0$ independent of $t$ such that
$$
\int_P u^2(t) ~ p(z) d \mu < C
$$
for all $t \in [0, T)$. Since $p(z) > 0$ on $\bar{P}$, there exists a constant $C > 0$ independent of $t$ such that
$$
\int_P u^2(t) ~ d \mu < C
$$
\end{proof}

Let $Rm$ be the Riemannian curvature of $X$ and $Rm_{\CC \PP}$ be the Riemannian curvature restricted on $\CC \PP^2$. Since $\CC \PP^2$ is a totally geodesic submanifold of $X$, one has $|Rm_{\CC \PP^2}|^2 (u) = \sum_{ijkl} u^{ij}_{~ kl} u^{kl}_{~ij} \le |Rm|^2(u)$. Since $\int_X |Rm|^2(t) ~ \omega^n(t)$ is decreases along the Calabi flow, there exists a constant $C > 0$ independent of $t$ such that
$$
\int_P \sum_{ijkl} \left(u^{ij}_{~kl} u^{kl}_{~ij}\right) (t)~ p(z) d \mu < C. 
$$
Thus one has 
\begin{lemma}
There exists a constant $C > 0$ independent of $t$ such that
$$
\int_P \sum_{ijkl} \left(u^{ij}_{~kl} u^{kl}_{~ij}\right) (t)~ d \mu < C, \quad \int_P \left(\sum_{ij} u^{ij}_{~ij} \right)^2 (t)~ d \mu < C,
$$
for all $t \in [0, T)$.
\end{lemma}

Similar to the calculations in \cite{CHS}, one obtains
\begin{cor}
There exists a constant $C > 0$ independent of $t \in [0, T)$ such that
\begin{align*}
\bigg| \int_P u_{ij}(0) u^{ij}(t) ~ d\mu \bigg| \le C + 2 \bigg| \int_{\partial P} (u(t) -u(0)) ~ d\sigma \bigg|,\\
\bigg| \int_P u_{ij}(t) u^{ij}(0) ~ d\mu \bigg| \le C + 2 \bigg| \int_{\partial P} (u(t) -u(0)) ~ d\sigma \bigg|.
\end{align*}
\end{cor}

To control $\bigg| \int_{\partial P} (u(t) -u(0)) ~ d\sigma \bigg|$, one needs to provide an explicit formula of the derivative of the Calabi energy. By direct calculations, one has

\begin{lemma}
\label{difference_Calabi_energy}
$$
\int_P \left( R^2(0) - R^2(t_0) \right) ~ p(z) d\mu = 2 \int_0^{t_0} \int_P  u^{ir}(t) u^{js}(t) R(t)_{ij} R(t)_{rs} ~ p(z) d \mu dt
$$
\end{lemma}

\begin{proof}
Our lemma follows from the following calculations:
\begin{align*}
& \frac{\partial}{\partial t} \int_P R^2(t) ~ p(z) d\mu \\
=& - 2 \int_P R(t) \left( \frac{\partial^2}{\partial z_r \partial z_s} \left( p(z) u^{ir}(t) R(t)_{ij} u^{sj}(t) \right)\right) d \mu\\
=& - 2 \int_P u^{ir}(t) u^{js}(t) R(t)_{ij} R(t)_{rs} ~ p(z) d \mu.
\end{align*}
\end{proof}

An immediate corollary is 

\begin{cor}
There exists a constant $C > 0$ independent of $t$ such that for any $t \in [0, T)$, one has
$$
\bigg| \int_{\partial P} u(t) ~ d \sigma \bigg| < C.
$$
\end{cor}

\begin{proof}
Note that
\begin{align*}
& \int_P u^{ir}(t) u^{js}(t) R(t)_{ij} R(t)_{rs} ~ p(z) d \mu\\
\ge & C \int_P \left( u^{ij}(t) R(t)_{ij} \right)^2 ~ p(z) d \mu\\
\ge & C \left(\int_P u^{ij}(t) R(t)_{ij} ~ p(z) d \mu \right)^2\\
\end{align*}

Integrating by parts, one obtains

\begin{align*}
& \int_P u^{ij}(t) R(t)_{ij} ~ p(z) d \mu \\
=& \int_P \left( p(z) u^{ij}(t) \right)_{ij} R(t) ~ d \mu - \int_{\partial P} R(t) ~ p(z) d \sigma \\
=& \int_P \left( \frac{Scal_S}{\langle p_S, z \rangle + c_S} - R(t) \right) R(t) ~ p(z) d \mu  - \int_{\partial P} R(t) ~ p(z) d \sigma \\
\end{align*}

Thus we conclude that for any $t_0 \in [0, T)$, one has
$$
\bigg| \int_0^{t_0} \int_{\partial P} R(t) ~ p(z) d\sigma dt \bigg| < C.
$$
Hence
$$
\bigg| \int_{\partial P} \left( u(t) - u(0) \right) ~p (z) d \sigma \bigg| < C.
$$
The uniform control of $\int_P u^2(t) ~ d \mu$ for any $t \in [0, T)$ provides a uniform lower bound of the minimum of $u(t)$. So one can obtain
$$
\bigg| \int_{\partial P} u(t) d \sigma \bigg| < C,
$$
for any $t \in [0, T)$.
\end{proof}

Combing the above results, one obtains

\begin{cor}
There exists a constant $C > 0$ such that for any $t \in [0, T)$, one has
$$
\bigg| \int_P u_{ij}(0) u^{ij}(t) ~ d\mu \bigg| \le C,\quad \bigg| \int_P u_{ij}(t) u^{ij}(0) ~ d\mu \bigg| \le C.
$$
\end{cor}

Now we are ready to give a proof of \autoref{distance}.
\begin{proof}[Proof of \autoref{distance}]
Note that we have uniformly controlled the following quantities:
\begin{align*}
\int_P u^2(t) ~ d \mu < C \\
\int_P \left(u^{ij}_{~kl} u^{kl}_{~ij}\right) (t) ~ d \mu < C \\
\bigg| \int_P Trace(u^{ij}(t)) ~ d \mu \bigg| < C.
\end{align*}

Thus following the proof of Theorem (1.3) in \cite{CHS}, we complete our proof.
\end{proof}

To obtain the uniform interior estimates along the Calabi flow $u(t), ~ t \in [0, T)$, we also need the following proposition:

\begin{prop}
\label{blow_up}
For every $\epsilon > 0$, there exists $C(\epsilon) > 0$ such that for any $t \in [0, T)$ and any $x \in P_\epsilon$, one has
$$
Q(t, x) d^2_{u(t)} (x, \partial P_\epsilon) < C(\epsilon),
$$
where $Q(t, x) =\left( |Rm| + |\nabla Rm|^{2/3} + |\nabla^2 Rm|^{1/2}\right)(t, x).$
\end{prop}

We need to modify the proof of Theorem 1.2 in \cite{CHS} to fit into our situation. Our proof follows closely the section 4 in \cite{CHS}. Suppose the conclusion of \autoref{blow_up} is not true, then there exists a sequence $(t_i, x_i) \in [0, T) \times P_\epsilon$ such that 
$$
Q(t_i, x_i) d^2_{u(t_i)}(x_i, \partial P_\epsilon) = \max_{t \le t_i, ~ x \in P_\epsilon} Q(t, x) d^2_{u(t)} (x, \partial P_\epsilon)
$$
and
$$
\lim_{i \to \infty} Q(t_i, x_i) d^2_{u(t_i)}(x_i, \partial P_\epsilon) = \infty.
$$

Let $\lambda_i = Q(t_i, x_i)$ and we rescale the Calabi flow by $\lambda_i$. To do that, first we need to rescale the symplectic potential $u(t, x)$ such that the new symplectic potential $u^{(i)}(t)$ on $P^{(i)} = \lambda_i P$ is
$$
u_i(t, x) = \lambda_i u \left( \frac{t+t_i}{\lambda_i^2}, \frac{x+x_i}{\lambda_i} \right).
$$
We also need to rescale the metric $g_S$ by the factor $\lambda_i$. After rescaling the Calabi flow by $\lambda_i$, we obtain a sequence of Calabi flows $g^{(i)}(t)$. 
\begin{lemma}
There exists a constant $C > 0$ independent of $t$ and $g^{(i)}$ such that for any $t \in [-1, 0]$ and any sufficiently large $i$, we have
$$
d_{u^{(i)}(t)}(0, \partial P^{(i)}_\epsilon) \ge C d_{u^{(i)}(0)}(0, \partial P^{(i)}_\epsilon).
$$
\end{lemma}

\begin{proof}
Note that in our case, the Calabi energy of $g^{i}(t)$, $$Ca(g^{(i)}(t)) = \lambda_i^m Ca \left(g \left( \frac{t+t_i}{\lambda_i^2} \right) \right),$$
where $m = \dim_\CC S$. But
$$
\int_P R^2 ~ p(z) d \mu, \quad \int_P |Rm|^2 ~ p(z) d \mu
$$
are scaling invariant. 

Let us suppress the index $i$ of the Calabi flow $g^{(i)}(t)$. Applying \autoref{difference_Calabi_energy}, we have
$$
\int_{-1}^0 u^{ij}(t) u^{kl}(t) R_{ik}(t) R_{jl}(t) ~ d \mu d t \le C,
$$
where $C$ is a positive constant independent of $t$ and $g^{(i)}$. Since $\int_P |Rm|^2 ~ p(z) d \mu < C$, one has $\int_P u^{ij}_{~kl} u^{kl}_{~ij} ~ d \mu < C$, again $C$ is some constant independent of $t$ and $g^{(i)}$. Thus one can follow the proof of Proposition 4.1 in \cite{CHS} to finish the proof here.
\end{proof}

Next we need to construct the cutoff function $f_t$ on the polytope $P$ for $t \in [-1, 0]$. For any $x \in P$, let $d_t(x)$ be the Riemannian distance from the origin to $x$. Then $f_t$ needs to satisfy the following properties:
\begin{enumerate}
\item $f_t(d_t(x)) = 1$, for all $x \in P$ such that $d_t(x) \le 1/2$.

\item $f_t(d_t(x)) = 0$, for all $x \in P$ such that $d_t(x) \ge 1$.

\item $|\nabla f_t| \le C f_t^{\frac{a-1}{a}}, ~ \quad |\triangle f_t| \le C f_t^{\frac{a-1}{a}}$, where $C \in \RR^+, ~ a \in \ZZ^+$ are constants independent of $t$ and $g^{(i)}$.
\end{enumerate}

Our construction of $f_t$ is almost identical to the construction of $f_t$ in Section 4.1 in \cite{CHS}. Next we will need to obtain the following inequality:

\begin{prop}
$$
\frac{\partial}{\partial t} \int_P f_t |\nabla^k Rm_t|_t^2 ~ p(z) d\mu \le -\frac{1}{4} \int_P f_t |\nabla^{k+2} Rm_t|_t^2 ~ p(z) d \mu + C,
$$
where $t \in [-1, 0]$ and $C$ is some constant independent of $t$ and $g^{(i)}$.
\end{prop}
Notice that the Calabi energy is not scaling invariant in our case, but $\int_P |Rm|^2 ~ p(z) d \mu$ is scaling invariant. Thus our proof is almost identical to Theorem 4.7 in \cite{CHS}. Similar to Corollary 4.8 in \cite{CHS}, we obtain

\begin{cor} At $t=0$, we have
$$
\int_P f |\nabla^k Rm|^2 ~ p(z) d \mu < C(k),
$$
where $k \in \ZZ$ and $C(k)$ is a constant depending on $k$ but independent of $g^{(i)}$.
\end{cor}

Recall that $\CC \PP^2$ is a totally geodesic submanifold of $X$. Let $\overline{g}, ~\overline{\nabla}, ~ \overline{Rm}$ be the induced Riemannian metric, covariant derivative and Riemann curvature tensor on $\CC \PP^2$. Thus we have
$$
|\overline{\nabla}^k \overline{Rm}|_{\overline{g}} \le |\nabla^k Rm|_g.
$$
So we obtain

\begin{cor} At $t=0$, we have
$$
\int_P f |\overline{\nabla}^k \overline{Rm}|_{\overline{g}}^2 ~ p(z) d \mu < C(k),
$$
where $k \in \ZZ$ and $C(k)$ is a constant depending on $k$ but independent of $g^{(i)}$.
\end{cor}

Applying the arguments in the appendix of \cite{FH}, we obtain

\begin{cor}
For any $p \in B_g(0, 1/2)$ where $B_g(0, 1/2)$ denotes the geodesic ball of $P$ centered at the origin with radius $1/2$, we have
$$
|\overline{\nabla}^k \overline{Rm}|_{\overline{g}}(p) < C(k),
$$
again $C(k)$ is a constant depending on $k$ but independent of $g^{(i)}$.
\end{cor}

\begin{proof}[Proof of \autoref{blow_up}]

Our proof is similar to the proof of Theorem 1.2 in \cite{CHS}. Recall that we have obtained a sequence of metric $\overline{g}^{(i)}$ on $\CC \PP^2$. We want to study the corresponding symplectic potentials $u^{(i)}$ on $P_i$. We have two cases:

Case I. By possibly passing to a subsequence, there exists $p_i \in B_i(0, 1/2)$ such that $|\overline{Rm}^{(i)}(p_i)| \ge C > 0$ for some constant $C$ independent of $i$. Then $u^{(i)}$ converges to a smooth, strictly convex function $u$ on $\RR^2$ with 
$$
u^{ij}_{~ij} = 0, \quad u^{ij}_{~kl} u^{kl}_{~ij} (p_\infty) \ge C,
$$
where $p_\infty$ is the limit of $p_i \in \RR^2$. However, in this case $u$ must be a quadratic function, a contradiction. 

Case II. Lifting to the tangent space of $\CC \PP^2$, $|\overline{Rm}^{(i)}|$ converges uniformly to $0$ on $B_i(0, 1/2)$. However, at the limit, we cannot have $|\overline{Rm}^{(\infty)}| = 0$ on $B_\infty(0, 1/2)$ and  $Q_{\overline{g}^\infty} (0) = 1$.

\end{proof}

Combining \autoref{distance} and \autoref{blow_up}, using the arguments of the proof of Theorem 1.1 in \cite{CHS}, we obtain \autoref{interior}.

\section{$C^0$-estimate}

In this section, we focus on $(X, J) = P(E) \to \Sigma$, where $E = \mathcal{O} \oplus \mathcal{L}_1 \oplus \mathcal{L}_2$ and $\mathcal{L}_1, \mathcal{L}_2$ are holomorphic line bundles over a compact complex curve $\Sigma$.

\subsection{Bisectional curvature}

Let us calculate the norm of bisectional curvature of $x \in X$ with $z(x) \in P^0$. First of all, let $\eta_0 = x_0 + \sqrt{-1} y_0$ be a local holomorphic coordinate in a domain $\Omega \subset \Sigma$. Then on $\Omega$, we can write $$\omega_\Sigma = i \partial \bar{\partial} \varphi_\Sigma,$$

where $\varphi_\Sigma(z) = |\eta_0|^2 + a |\eta_0|^4 + o(|\eta_0|^5)$ and $a = -\frac{Scal_\Sigma}{2}$ is a constant. By the Legendre transformation, the dual coordinates of $z_1, z_2$ are $$\xi_1 := \frac{\partial u}{\partial z_1}, ~ \xi_2 := \frac{\partial u}{\partial z_2}$$ and the dual of the symplectic potential $u$ is $$\varphi_{\CC \PP^2} (\xi_1, \xi_2) := \xi_1 z_1 + \xi_2 z_2 - u(z_1, z_2).$$ Then $$\eta_i = \xi_i - p_i \varphi_\Sigma + \sqrt{-1} t_i, ~ i = 1, 2$$ are the holomorphic coordinates, where $\varphi_\Sigma$ is the obvious pullback function via $\pi : X \to \Sigma$. Also the K\"ahler potential on $X^0 = \Omega \times P^0 \times \TT^2$ is \[
  \varphi = c_\Sigma \varphi_\Sigma  + 2 \varphi_{\CC \PP^2},
\]
where $c_\Sigma$ is some constant parameterizing the K\"ahler class.

Next we want to write down the K\"ahler metric in our coordinates $\eta_0, \eta_1, \eta_2$. Our old real coordinate system on $\Omega \times P^0$ is $x_0, y_0, \xi_1, \xi_2$ and our new real coordinate system is $x_0, y_0, \zeta_1 := \xi_1 - p_1 \varphi_\Sigma, \zeta_2 := \xi_2 - p_2 \varphi_\Sigma$. Thus we have

\begin{align*}
  & ~ \left(
  \begin{array}[]{cccc}
    \frac{\partial x_0}{\partial x_0} & \frac{\partial x_0}{\partial y_0} & \frac{\partial x_0}{\partial \zeta_1} & \frac{\partial x_0}{\partial \zeta_2}\\
    \vdots & \vdots & \vdots & \vdots \\
    \frac{\partial \xi_2}{\partial x_0} & \frac{\partial \xi_2}{\partial y_0} & \frac{\partial \xi_2}{\partial \zeta_1} & \frac{\partial \xi_2}{\partial \zeta_2}
  \end{array}
  \right) \\
  = & ~ \left( 
  \begin{array}[]{cccc}
    \frac{\partial x_0}{\partial x_0} & \frac{\partial x_0}{\partial y_0} & \frac{\partial x_0}{\partial \xi_1} & \frac{\partial x_0}{\partial \xi_2}\\
    \vdots & \vdots & \vdots & \vdots \\
    \frac{\partial \zeta_2}{\partial x_0} & \frac{\partial \zeta_2}{\partial y_0} & \frac{\partial \zeta_2}{\partial \xi_1} & \frac{\partial \zeta_2}{\partial \xi_2}
  \end{array}
  \right)^{-1} \\
  = & ~ \left( 
  \begin{array}[]{cccc}
    1 & 0 & 0 & 0 \\
    0 & 1 & 0 & 0 \\
    -p_1 \varphi_{\Sigma, ~ x_0} & -p_1 \varphi_{\Sigma, ~ y_0} & 1 & 0 \\
    -p_2 \varphi_{\Sigma, ~ x_0} & -p_2 \varphi_{\Sigma, ~ y_0} & 0 & 1 \\
  \end{array}
  \right)^{-1}\\
  = & ~ \left( 
  \begin{array}[]{cccc}
    1 & 0 & 0 & 0 \\
    0 & 1 & 0 & 0 \\
    p_1 \varphi_{\Sigma, ~ x_0} & p_1 \varphi_{\Sigma, ~ y_0} & 1 & 0 \\
    p_2 \varphi_{\Sigma, ~ x_0} & p_2 \varphi_{\Sigma, ~ y_0} & 0 & 1 \\
  \end{array}
  \right)\\
\end{align*}

At a point $x$, direct calculations show

\begin{align*}
  \omega(x) = \sqrt{-1} \left( 2 p(z) d \eta_0 \wedge d \bar{\eta}_0 + \frac{1}{2} H_{ij} \eta_i \wedge \bar{\eta}_j \right), ~ i,j \ge 1,
\end{align*}

where $H_{ij} = \frac{\partial^2 \varphi_{\CC \PP^2}}{\partial \xi_i \partial \xi_j}$. 

\begin{align*}
  \varphi_{\eta_0 \bar{\eta}_0 \eta_0} =&~ 0 \\
  \varphi_{\eta_0 \bar{\eta}_i \eta_0} =&~ 0, ~ i \ge 1 \\
  \varphi_{\eta_0 \bar{\eta}_0 \eta_i} =&~ p_1 H_{1i} + p_2 H_{2i}, ~ i \ge 1 \\
  \varphi_{\eta_0 \bar{\eta}_i \eta_j} =&~ 0, ~ i, j \ge 1 \\
  \varphi_{\eta_i \bar{\eta}_0 \eta_j} =&~ 0, ~ i, j \ge 1 \\
  \varphi_{\eta_i \bar{\eta}_j \eta_k} =&~ \frac{1}{4} H_{ijk}, ~ i, j, k \ge 1 \\
\end{align*}

and

\begin{align*}
  \varphi_{\eta_0 \bar{\eta}_0 \eta_0 \bar{\eta}_0} = &~ 4a p(z) + 4 \left(p_1^2 H_{11} + 2 p_1 p_2 H_{12} + p_2^2 H_{22}\right)\\
  \varphi_{\eta_0 \bar{\eta}_0 \eta_0 \bar{\eta}_i} = &~ 0, ~ i \ge 1 \\
  \varphi_{\eta_0 \bar{\eta}_0 \eta_i \bar{\eta}_j} = &~ \frac{1}{2} \left( p_1 H_{1ij} + p_2 H_{2ij} \right) \\
  \varphi_{\eta_0 \bar{\eta}_i \eta_0 \bar{\eta}_j} = &~ 0\\
  \varphi_{\eta_0 \bar{\eta}_i \eta_j \bar{\eta}_k} = &~ 0\\
  \varphi_{\eta_i \bar{\eta}_j \eta_k \bar{\eta}_l} = &~ \frac{1}{8} H_{ijkl}.\\
\end{align*}

The bisectional curvature can be expressed as:

\begin{align*}
  Rm_{i \bar{j} k \bar{l}} = - \varphi_{i \bar{j} k \bar{l}} + g^{s\bar{t}} \varphi_{i \bar{l} \bar{t}} \varphi_{\bar{j} k s}.
\end{align*}

We have

\begin{align*}
  Rm_{0 \bar{0} 0 \bar{0}} =& ~  - 4a p(z) - 2 \left( p_1^2 H_{11} + 2 p_1 p_2 H_{12} + p_2^2 H_{22} \right)\\
  Rm_{0 \bar{0} 0 \bar{i}} =& ~ 0, ~ i \ge 1 \\
  Rm_{0 \bar{0} i \bar{j}} =& ~ - \frac{1}{2} \left( p_1 H_{1ij} + p_2 H_{2ij} \right) + \frac{1}{2 p(z)} (p_1 H_{1i} + p_2 H_{2i}) (p_1 H_{1j} + p_2 H_{2j}), ~ i, j \ge 1 \\
  Rm_{0 \bar{i} 0 \bar{j}} =& ~ 0 , ~ i, j \ge 1\\
  Rm_{0 \bar{i} j \bar{k}} =& ~ 0 , ~ i, j, k \ge 1\\
  Rm_{i \bar{j} k \bar{l}} =& ~ \frac{1}{8} \left(-H_{ijkl} + H^{st}H_{ilt}H_{jks} \right) , ~ i, j, k, l \ge 1\\
\end{align*}

So the Ricci curvature is

\begin{align*}
  Ric_{0\bar{0}} =& ~ -2a - H^{ij} \left(p_1 H_{1ij} + p_2 H_{2ij}\right), ~ i, j \ge 1.\\
  Ric_{0 \bar{i}} =& ~ 0, ~ i \ge 1. \\
  Ric_{i \bar{j}} =& ~  \frac{1}{4} H^{kl} \left( -H_{ijkl} + H^{st} H_{ilt} H_{jks} \right), ~ i, j, k, l \ge 1. \\
\end{align*}

and the scalar curvature is

\begin{align*}
  R = \frac{-2a}{p(z)} - \frac{1}{p(z)} \frac{\partial^2}{\partial z_r \partial z_s} \left( p(z) u^{rs} \right).
\end{align*}

By definition, the norm of the bisectional curvature at $x$ is

\begin{align*}
  |Rm(x)|^2 = & \frac{1}{4 p^4(z)} \left( 2a p(z) + p_1^2 H_{11} + 2 p_1 p_2 H_{12} + p_2^2 H_{22} \right)^2 + \\
  & \frac{1}{4 p(z)^2} H^{ik} H^{jl} \left(- p_1 H_{1ij} - p_2 H_{2ij} + \frac{1}{p(z)} (p_1 H_{1i} + p_2 H_{2i})(p_1 H_{1j} + p_2 H_{2j}) \right)\\
  & \left(- p_1 H_{1kl} - p_2 H_{2kl} + \frac{1}{p(z)} (p_1 H_{1k} + p_2 H_{2k})(p_1 H_{1l} + p_2 H_{2l}) \right) + |Rm(x)_{\CC \PP^2}|^2
\end{align*}

\subsection{Intial metric}

Let us choose the initial metric of the Calabi flow on the $\CC \PP^2$ fiber to be the Fubini-Study metric and we will estimate the norm of bisectional curvature.

Let the Delzant triangle $P$ to be the triangle of three vertex $(-1, -1), (-1, 2)$ and $(2, -1)$. The corresponding K\"ahler class of $P$ will be $c_1 (\mathcal{O}(3))$. Let the symplectic potential be
\begin{align*}
  v(x,y) = \frac{1}{2} \left( (x+1) \ln (x+1) + (y+1) \ln(y+1) + (1-x-y) \ln (1-x-y) \right).
\end{align*}
Then

\begin{align*}
  \left( D^2 v(x,y) \right) =& ~ \frac{1}{2} \left(
  \begin{array}[]{ll}
    \frac{2-y}{(1+x)(1-x-y)} & \frac{1}{1-x-y} \\
    \frac{1}{1-x-y} & \frac{2-x}{(1+y)(1-x-y)}
  \end{array}
  \right)\\
  \det\left( D^2 v(x,y) \right) =& ~ \frac{3}{4} \frac{1}{(1+x)(1+y)(1-x-y)} \\
  \left( D^2 v(x,y) \right)^{-1} =& ~ \frac{2}{3} \left( 
  \begin{array}[]{ll}
    (2-x)(1+x) & -(1+x)(1+y) \\
    -(1+x)(1+y) & (2-y)(1+y)
  \end{array}
  \right)
\end{align*}

For any point $(x,y) \in P$, we have

\begin{align*}
  |Rm|^2_{\CC \PP^2} (x,y) =& \frac{1}{4} \sum_{i,j,k,l} v^{ij}_{~kl} v^{kl}_{~ij} = \frac{4}{3} \\
  R_{\CC \PP^2}(x,y) = & - \sum_{ij} v^{ij}_{~ij} = 4
\end{align*}

Let $\tilde{x}$ be any point in $X$. Notice that
\begin{align*}
  H_{ij}(z(\tilde{x})) = \left( D^2 v(z(\tilde{x})) \right)^{-1}.
\end{align*}

and 

\begin{align*}
  v^{xx} < 3, ~ v^{yy} < 3, ~ |v^{xy}| < 6.
\end{align*}

Together with $p_1 \ge p_2$, we have 
\begin{align*}
  &  \frac{1}{4 p^4(z(\tilde{x}))} \left( 2a p(z(\tilde{x})) + p_1^2 H_{11}(\tilde{x}) + 2 p_1 p_2 H_{12}(\tilde{x}) + p_2^2 H_{22}(\tilde{x}) \right)^2 \\
  \le ~ & \frac{Scal_\Sigma^2}{p^2(z(\tilde{x}))} + \frac{1}{p^4(z(\tilde{x}))} (9 p_1^4 + 72 p_1^2 p_2^2 + 9 p_2^4) \\
  \le ~ & \frac{Scal_\Sigma^2}{p^2(z(\tilde{x}))} + \frac{90 p_1^4}{p^4(z(\tilde{x}))}
\end{align*}

By the Legendre transform, we have

\begin{align*}
  H^{ik} \frac{\partial}{\partial \xi_k} = \frac{\partial}{\partial z_i}.
\end{align*}

Thus

\begin{align*}
  & \bigg| \frac{1}{4p(z)^2} \sum_{i,j,k,l} G_{ik} G_{jl} \left(- p_1 H_{1ij} - p_2 H_{2ij} + \frac{1}{p(z)} (p_1 H_{1i} + p_2 H_{2i})(p_1 H_{1j} + p_2 H_{2j}) \right)\\
  & \left(- p_1 H_{1kl} - p_2 H_{2kl} + \frac{1}{p(z)} (p_1 H_{1k} + p_2 H_{2k})(p_1 H_{1l} + p_2 H_{2l}) \right) \bigg| \\
  = ~ & \bigg| \frac{1}{4 p(z)^2} \sum_{i,l} \left( -p_1 \frac{\partial H_{1i}}{\partial z_l} - p_2 \frac{\partial H_{2i}}{\partial z_l} + \frac{1}{p(z)}(p_1 H_{1i} + p_2 H_{2i})(p_1 \delta^1_l + p_2 \delta^2_l )\right) \\
  & \left( -p_1 \frac{\partial H_{1l}}{\partial z_i} - p_2 \frac{\partial H_{2l}}{\partial z_i} + \frac{1}{p(z)}(p_1 \delta^1_i + p_2 \delta^2_i)(p_1 H_{1l} + p_2 H_{2l} )\right) \bigg| \\
  \le ~ & \frac{1}{p(z)^2} \left(2 p_1 + 2 p_2 + \frac{6 \left( p_1 + p_2 \right)^2}{p(z)} \right)^2 \\
  \le ~ & \frac{1}{p(z)^2} \left(4 p_1 + \frac{24 p_1^2}{p(z)} \right)^2.
\end{align*}

In the last inequality of the above calculations, we use the fact that

\begin{align*}
  |v^{xx}_{~x}| \le 2, ~ |v^{xy}_{~y}| \le 2, ~ |v^{xy}_{~x}| \le 2, ~ |v^{yy}_{~y}| \le 2.
\end{align*}

Finally we have the following inequality

\begin{align}
  \label{Control_Rm}
  |Rm(\tilde{x})|^2 \le \frac{1}{p^2(z(\tilde{x}))} \left( Scal_\Sigma^2 + \frac{90 p_1^4}{p^2(z(\tilde{x}))} + \left(4 p_1 + \frac{24 p_1^2}{p(z(\tilde{x}))} \right)^2 \right) + \frac{4}{3}.
\end{align}

\subsection{Sobolev constant}

Recall that for a closed Riemannian manifold $(X, g)$ with real dimension $m$. The Sobolev constant is a constant $C_s$ depending on $(X, g)$ such that for any smooth function $f$, we have

\begin{equation}
  \label{Sobolev_inequality}
  \| f \|_{L^{\frac{2m}{m-2}}} \le C_s \left( \| \nabla f \|_{L^2} + Vol^{-\frac{1}{m}} \| f \|_{L^2} \right),
\end{equation}
where $Vol$ is the volume of $(X, g)$. It is easy to check that the Sobolev constant is scaling invariant.

The idea of obtaining a concrete value of the Sobolev constant $C_s$ for a K\"ahler surface $(X, J, \omega)$ with a cscK metric goes back to Tian. See also Tian-Viaclovsky \cite{TV1, TV2} and Chen-LeBrun-Weber \cite{CLW}. Later, Chen-He \cite{ChenHe2} generalizes this idea to control the Sobolev constant along the Calabi flow. The arguments presented in this subsection are known in the literature. We put them here for the convenience of the readers.

To control the Sobolev constant, one needs to relate it to the Yamabe constant. Recall that for a closed 4-manifold $(X, g_0)$, the Yamabe constant of the conformal class $[g_0]$ is defined as

\begin{align*}
  Y_{[g_0]} = \inf_{g \in [g_0]} \frac{\int_X R_g ~ d \mu_g}{\sqrt{Vol(X, g)}}.
\end{align*}

To establish the inequality \eqref{Sobolev_inequality} for any smooth function $f$, we consider a new metric $g = f^2 g_0$ conformal to $g_0$. Then the scalar curvature of $g$ is

\begin{align*}
  f^3 R_g = (- 6\triangle_{g_0} + R_{g_0}) f.
\end{align*}

Then

\begin{align}
  \label{Yamabe}
  Y_{[g_0]} \le \frac{\int_X (6 |\nabla f|_{g_0}^2 + R_{g_0}f^2) ~ d \mu_{g_0}}{\sqrt{\int_X f^4 ~ d \mu_{g_0}}}.
\end{align}

In our case, we need to turn the above inequality \eqref{Yamabe} to a Sobolev inequality in $\CC \PP^2$. The first thing we would do is to control the Yamabe constant $Y_{[g_0]}$, where $g_0$ is a K\"ahler metric in $c_1(\mathcal{O}_{\CC \PP^2}(3))$. For any $g$ in the conformal class $[g_0]$, the Gauss-Bonnet formula tells us that

\begin{align*}
  \frac{1}{8 \pi^2} \int_{\CC \PP^2} \left( |W_+|^2 + |W_{-}|^2 + \frac{R^2}{24} - \frac{|Ric_0|^2}{2} \right) ~ d \mu_{g} = \chi(\CC \PP^2),
\end{align*}

where 

\begin{enumerate}
  \item $W_+$ is the self-dual Weyl curvature of $g$.
  \item $W_{-}$ is the anti-self-dual Weyl curvature of $g$.
  \item $R$ is the scalar curvature of $g$. 
  \item $Ric_0$ is the trace-free Ricci curvature of $g$.
  \item $\chi$ is the Euler-characteristic.
\end{enumerate}

And the Hirzebruch signature formula tells us that

\begin{align*}
  \frac{1}{12 \pi^2} \int_X \left( |W_+|^2 - |W_{-}|^2 \right) ~ d \mu_{g}  = \tau(\CC \PP^2),
\end{align*}
where the signature $\tau(\CC \PP^2) = b_+(\CC \PP^2) - b_{-} (\CC \PP^2).$

We have

\begin{align}
  \label{eq_Y}
  \frac{1}{4 \pi^2} \int_{\CC \PP^2} \left(\frac{R^2}{24} + 2 |W_+|^2\right) ~ d \mu_{g} \ge \left( 2 \chi + 3 \tau \right) (\CC \PP^2) = c_1^2(\CC \PP^2).
\end{align}

The last equality holds because $\CC \PP^2$ admits an almost complex structure. By the celebrated work of Trudinger, Aubin, and Schoen, there exists a metric $g_Y \in [g_0]$ such that 

\begin{align*}
  Y_{[g]} = \frac{\int_{\CC \PP^2} R(g_Y) ~ d \mu_{g_Y}}{\sqrt{\int_{\CC \PP^2} ~ d \mu_{g_Y}}},
\end{align*}

and $R(g_Y)$ is a constant on $\CC \PP^2$.  The inequality \eqref{eq_Y} becomes

\begin{align*}  
\frac{1}{4 \pi^2} \int_{\CC \PP^2} \left(\frac{R^2(g_Y)}{24} + 2 |W_+(g_Y)|^2\right) ~ d \mu_{g_Y} \ge c_1^2(\CC \PP^2).
\end{align*}

Thus

\begin{align*}
  Y_{[g_0]}^2 \ge 96 \pi^2 c_1^2(\CC \PP^2) - 48 \int_{\CC \PP^2} |W_+(g_Y)^2| ~ d \mu_{g_Y}.
\end{align*}

Since 

\begin{align*}
  \int_{\CC \PP^2} |W_+|^2 ~ d \mu
\end{align*}

is a conformal invariant, we have

\begin{align*}
  Y_{[g_0]}^2 \ge 96 \pi^2 c_1^2(\CC \PP^2) - 48 \int_{\CC \PP^2} |W_+(g_0)^2| ~ d \mu_{g_0}.
\end{align*}

Since $g_0$ is K\"ahler, we have

\begin{align*}
  \int_{\CC \PP^2} |W_+(g_0)|^2 ~ d \mu_{g_0} = \int_{\CC \PP^2} \frac{R(g_0)^2}{24} ~ d \mu_{g_0}.
\end{align*}

Thus

\begin{align*}
  Y_{[g_0]}^2 \ge 96 \pi^2 c_1^2(\CC \PP^2) - 2 \int_{\CC \PP^2} R(g_0)^2 ~ d \mu_{g_0}.
\end{align*}

Using the proof of Lemma 6.1 in \cite{ChenHe2}, we know that if 

\begin{align}
  \label{eq_cs}
  96 \pi^2 c_1^2(\CC \PP^2) - 2 \int_{\CC \PP^2} R(g_0)^2 ~ d \mu_{g_0} \ge \int_{\CC \PP^2} (R(g_0) - \underline{R})^2 ~ d \mu_{g_0},
\end{align}

then $Y_{[g_0]}$ is positive and 

\begin{align*}
  C_s \le \max \left(6, \underline{R}\sqrt{Vol(\CC \PP^2, g_0)} \right) \left( Y_{[g_0]} - |R_{g_0} - \underline{R}|_{L^2} \right)^{-1}.
\end{align*}

Our previous calculations show that

\begin{align*}
  \underline{R} = 4, \quad Vol(\CC \PP^2, g_0) = \frac{1}{2} (2 \pi)^2 * \frac{9}{2}.
\end{align*}

Hence

\begin{align*}
  C_s \le 12 \pi \left( Y_{[g_0]} - |R_{g_0} - \underline{R}|_{L^2} \right)^{-1}.
\end{align*}

Now we want to derive an explicit bound for the Calabi energy of $g_0$ to satisfy \eqref{eq_cs}. By our choice of normalization, 

\begin{align*}
  c_1^2(\CC \PP^2) = \int_{\CC \PP^2} \left( \frac{\underline{R}}{2 \cdot 4 \pi} \omega_0 \right)^2 =  \frac{9}{2}. 
\end{align*}

Inequality \eqref{eq_cs} implies that

\begin{align*}
  96  \pi^2 > \int_{\CC \PP^2} (R(g_0) - \underline{R})^2 ~ d \mu_{g_0}.
\end{align*}

In summary, we have the following proposition:

\begin{prop}
  \label{prop_sc}
  For any K\"ahler metric $g_0 \in c_1(\mathcal{O}(3))$, if $Ca(g_0)$, i.e., the Calabi energy of $g_0$, is strictly less than $96 \pi^2$, then
  the Yamabe constant satisfies
  \begin{align*}
    Y_{[g_0]} > Ca(g_0),
  \end{align*}

  and the Sobolev constant satisfies

  \begin{align*}
     C_s \le 12 \pi \left( Y_{[g_0]} - Ca(g_0) \right)^{-1}.
  \end{align*}
\end{prop}

\subsection{Controlling $c_\Sigma$}

Let $c_0 = 12 p_1$ and the K\"ahler potential of an admissble K\"ahler metric on $X = P(E) \to \Sigma$ be

\begin{align*}
  \varphi = c_0 \varphi_\Sigma + 2 \varphi_{FS},
\end{align*}

here we choose the normalization such that $Scal_\Sigma = 1, 0 $ or $-1$. Then by Gauss-Bonnet theorem, 

\begin{align*}
  Vol(\Sigma, \varphi_\Sigma) = |4 \pi \chi(\Sigma)|.
\end{align*}

By the inequality \eqref{Control_Rm}, we have 

\begin{align*}
  |Rm(\tilde{x})|^2 \le \frac{1}{100 p_1^2} \left( 1 + \frac{90 p_1^2}{100} + \left(4 + \frac{24}{10} \right)^2 p_1^2 \right) + \frac{4}{3} < \frac{1}{2} + \frac{4}{3},
\end{align*}

for any $\tilde{x} \in X$. Then 

\begin{align*}
  \int_{X} |Rm|^2 ~ d \mu_{\varphi} <  4 \pi |\chi(\Sigma)| \frac{(2 \pi)^2}{3!} \frac{9}{2} 14 p_1 \left( \frac{1}{2} + \frac{4}{3} \right).
\end{align*}

Let $\varphi(t) = c_0 \varphi_\Sigma + \varphi_{\CC \PP^2}(t)$ be the Calabi flow on $X$ starting from $\varphi$. Then

\begin{align*}
  \int_{X} |Rm|^2(t) ~ d \mu_{\varphi(t)} <  4 \pi |\chi(\Sigma)| \frac{(2 \pi)^2}{3!} \frac{9}{2} 14 p_1 \left( \frac{1}{2} + \frac{4}{3} \right).
\end{align*}

Since $\CC \PP^2$ is a totally geodesic K\"ahler submanifold of $X$ of any admissible K\"ahler metric. We have

\begin{align*}
  |Rm|^2 (t) \ge |Rm|_{\CC \PP^2}^2(t).
\end{align*}

Hence

\begin{align*}
  \int_X |Rm|^2 (t) ~ d \mu_{\varphi(t)} \ge & \int_X |Rm|^2_{\CC \PP^2}(t) ~ d \mu_{\varphi(t)} \\
  \ge & 4 \pi |\chi(X)| \frac{(2 \pi)^2}{3!} 10 p_1 \int_P |Rm|^2_{\CC \PP^2}(t) ~ d \mu. 
\end{align*}

It gives us 

\begin{align*}
  \int_{P} |Rm|^2_{\CC \PP^2}(t) ~ d \mu < \frac{63}{10} \left( \frac{1}{2} + \frac{4}{3} \right).
\end{align*}

Since 

\begin{align*}
  \frac{(2\pi)^2}{2} \int_{P} |Rm|^2_{\CC \PP^2}(t) ~ d \mu - \frac{1}{4} Ca(\varphi_{\CC \PP^2}(t))
\end{align*}

is a constant and when $t=0$, $Ca(\varphi_{\CC \PP^2}(t)) = 0$. We have

\begin{align*}
  Ca(\varphi_{\CC \PP^2}(t)) =& ~  8 \pi^2 \int_P \left( |Rm|^2_{\CC \PP^2}(t) - |Rm|^2_{\CC \PP^2}(0) \right)~ d \mu \\
  < & ~ 45 \pi^2.
\end{align*}

By Proposition \eqref{prop_sc}, we have the following result:

\begin{prop}
\label{Sobolev_CP2}
  Let $c_\Sigma$ be any constant greater than $c_0 = 12 p_1$. Let $\varphi(t) = c_\Sigma \varphi_\Sigma + \varphi_{\CC \PP^2}(t)$ be the Calabi flow on $X$ starting from $\varphi = c_\Sigma \varphi_\Sigma + 2 \varphi_{FS}$. Then there exists a constant $C_s$ such that the Sobolev constant of the K\"ahler metrics corresponding to $\varphi_{\CC \PP^2}(t)$ on $\CC \PP^2$, i.e. $\omega_{\CC \PP^2}(t) = \sqrt{-1} \partial \bar{\partial} \varphi(t)$ on $P^0 \times \TT^2$ ,  is bounded by $C_s$.
\end{prop}

Now we are ready to prove \autoref{thm_proj_Sob} which is the following theorem:

\begin{thm}
\label{Sob}
There exists a constant $C > 0$ independent of $t$ such that for any $f \in C^\infty(\bar{P})$ and for any $t \in [0, T)$, one has
$$
\| f \|_{L^3(\omega(t))} \le C \left(\| f \|_{L^2(\omega(t))}  +  \| \nabla f \|_{L^2(\omega(t))} \right)
$$
\end{thm}

\begin{proof}
Notice that
$$
\| f \|_{L^3(\omega(t))} = \left( \int_P |f(z(t))|^3 ~ p(z(t)) d \mu \right)^{1/3} < C \left( \int_P |f(z(t))|^4 ~ d \mu \right)^{1/4}
$$
Also
\begin{align*}
& \| f \|_{L^2(\omega(t))}  +  \| \nabla f \|_{L^2(\omega(t))} \\
= & \left( \int_P |f|^2(z(t))~ p(z(t)) d \mu \right)^{1/2} + \left( \int_P |\nabla f|_{\omega_t}^2 (z(t))~ p(z(t)) d \mu \right)^{1/2}\\
\ge & C \left( \int_P |f|^2(z(t))~ d \mu \right)^{1/2} + \left( \int_P |\nabla f|_{\omega_t}^2 (z(t))~d \mu \right)^{1/2}.
\end{align*}

By \autoref{Sobolev_CP2},  the Sobolev constant of $\omega_{\CC \PP^2}(t) = \omega_{FS} + \sqrt{-1} \partial \bar{\partial} \varphi_{\CC \PP}^2(t), ~ t \in [0, T)$ is uniformly bounded on $\CC \PP^2$, the conclusion holds.
\end{proof}

By adapting the results in \cite{H1}, we obtain a proof of \autoref{c0}:
\begin{proof}[Proof of \autoref{c0}]
Once we are able to control the Sobolev constant along the Calabi flow, we can apply the techniques developed in \cite{H1} to get the uniform $C^0$-norm bounds. Let $\varphi_{\CC \PP^2} (t), ~ t \in [0, T)$ be the Calabi flow on $X$ with uniform Sobolev constants bounds. (Note that $\varphi_{\CC \PP^2} (t)$ does not satisfy the Calabi flow equation on $\CC \PP^2$.) Let $\omega(t) = \omega+ \sqrt{-1} \partial \bar{\partial} \varphi(t)$ be the corresponding one parameter family of Kaehler metrics on $X$. Let  $u(t)$ be the corresponding symplectic potentials on the Delzant polytope $P$. By \autoref{l2-estimate}, there exists a constant $C > 0$ independent of $t$ such that
$$
\int_P u^2(t) ~ d \mu < C
$$
for all $t \in [0, T)$. Then Lemma 3.3 in \cite{H1} shows that there exists a constant $C$ independent of $t$ such that
$$
\min_{x \in \bar{P}} u(t, x) > C,
$$
for all $t \in [0, T)$. By Proposition 3.4 in \cite{H1}, it implies that on the complex side, the $\max$ of relative Kaehler potentials $\varphi_{\CC \PP^2} (t)$ has a uniform upper bound, i.e.
$$
\max_{y \in X} \varphi_{\CC \PP^2} (t, y) < C
$$
where $t \in [0, T)$ and $C$ is some constant independent of $t$. Also by Corollary 3.7 in \cite{H1}, there exists some constant $C$ independent of $t$ such that for any $t \in [0, T)$, one has
$$
\max_{y \in X} \varphi_{\CC \PP^2} (t, y) > C
$$ 
Once we control $\max_{y \in X} \varphi_{\CC \PP^2} (t, y)$ uniformly, we can apply the results in the section 4 in \cite{H1} to show that there exists some constant $C > 0$ independent of $t$ such that
$$
\int_{X} |\varphi_{\CC \PP^2} (t)| ~ \omega^3 < C, \quad \int_{X} |\varphi_{\CC \PP^2} (t)| ~ \omega(t)^3 < C
$$

Since there exists a constant $C_1 > 0$ independent of $t$ such that for any $t \in [0, T)$, we have
$$
|\max_{y \in X} \varphi_{\CC \PP^2} (t, y)| < C_1.
$$
Applying Proposition 5.1 in \cite{H1} and \autoref{Sob}  in this section to $\varphi_{\CC \PP^2} (t, z) - C_1 - 1$, we obtain that there exists a constant $C > 0$ independent of $t$ such that
$$
\| \varphi_{\CC \PP^2} (t) \|_{L^2(\omega(t))} < C.
$$
Then Proposition 5.2 in \cite{H1} and  \autoref{Sob} in this section show that there exists a constant $C > 0$ independent of $t$ such that
$$
\| \varphi_{\CC \PP^2} (t) \|_{L^\infty} < C.
$$
\end{proof}

\begin{rmk}
One can also use Darvas's result \cite{Da, Da2} to obtain a proof of this corollary.
\end{rmk}

\end{document}